\documentclass[10pt]{article}
\usepackage[T1]{fontenc}
\usepackage[cp1250]{inputenc}
\usepackage{lmodern}

\usepackage[english]{babel}

\usepackage{latexsym}
\usepackage{amsmath}
\usepackage{indentfirst}
\usepackage{amsthm}
\usepackage{amssymb}
\usepackage{amsfonts}
\usepackage{stackrel}
\usepackage{dsfont}
\usepackage{tikz}
\usepackage{slashbox}
\usepackage[mathscr]{eucal}
\usepackage{authblk}
\usepackage{hyperref}

\newtheorem{thm}{Theorem}
\newtheorem{lemma}[thm]{Lemma}
\newtheorem{cor}[thm]{Corollary}

\newcommand{\reals}{\mathbb{R}}
\newcommand{\naturals}{\mathbb{N}}
\newcommand{\integers}{\mathbb{Z}}
\newcommand{\complex}{\mathbb{C}}

\newcommand{\eps}{\varepsilon}
\newcommand{\supp}{\text{supp}}

\newcommand{\Ffamily}{\mathcal{F}}

\begin{document}
\title{Natural proof of the characterization of relatively compact families in $L^p-$spaces on locally compact groups}
\author{Mateusz Krukowski}
\affil{Institute of Mathematics, \L\'od\'z University of Technology, \\ W\'ol\-cza\'n\-ska 215, \
90-924 \ \L\'od\'z, \ Poland \\ \vspace{0.3cm} e-mail: mateusz.krukowski@p.lodz.pl}
\maketitle

\begin{abstract}
In the paper we look for an elegant proof of the characterization of relatively compact families in $L^p-$spaces. At first glance, the suggested approach may seem convoluted and lengthy, but we spare no effort to argue that our proof is in fact more natural than the ones existent in the literature. The key idea is that the three properties which characterize relative compactness in $L^p-$spaces ($L^p-$boundedness, $L^p-$equicontinuity and $L^p-$equivanishing) are ``preserved'' (or rather ``inherited'') when the family $\Ffamily\subset L^p$ is convolved with a continuous and compactly supported function. This new family turns out to be relatively compact in $C_0-$space and it remains to be demonstrated that relative compactness in $C_0-$space implies relative compactness of the original family $\Ffamily.$ 
\end{abstract}

\vspace{0.2cm}

\section{Introduction}

Let us begin with a brief historical overview\footnote{See \cite{HancheOlsenHolden} for a much more detailed account.}: Andrey Kolmogorov ($1903-1987$) was arguably the first person, who succeeded in characterizing relatively compact families in $L^p(\reals^N),$ when $1 < p < \infty$ and all functions are supported in a common bounded set\footnote{See \cite{Kolmogorov} for the original paper by Kolmogorov.}. A year later (in $1932$) Jacob David Tamarkin ($1888-1945$) got rid of the second restriction\footnote{See \cite{Tamarkin} for the original paper by Tamarkin.} and in $1933$ Marcel Riesz ($1886-1969$), a younger brother of Frigyes Riesz ($1880-1956$), proved the general case\footnote{See \cite{Riesz} for the original paper by Riesz.}:

\begin{thm}
Let $1\leqslant p < \infty.$ A family $\Ffamily \subset L^p\left(\reals^N\right)$ is relatively compact if and only if
\begin{itemize}
	\item $\Ffamily$ is $L^p-$bounded, i.e. there exists $M>0$ such that 
	$$\forall_{f\in\Ffamily}\ \|f\|_p \leqslant M,$$
	\item $\Ffamily$ is $L^p-$equicontinuous, i.e. for every $\eps>0$ there exists $\delta>0$ such that 
	$$\forall_{\substack{|y|<\delta \\ f\in\Ffamily}}\ \int_{\reals^N}\ |f(x+y) - f(x)|^p\ dx \leqslant \eps,$$
	\item $\Ffamily$ is $L^p-$equivanishing, i.e. for every $\eps>0$ there exists $R>0$ such that 
	$$\forall_{f\in\Ffamily}\ \int_{|x| > R}\ |f(x)|^p \ dx \leqslant \eps.$$
\end{itemize}
\label{KolmogorovRiesz}
\end{thm}

In $1940$, Andr\'e Weil ($1906-1998$) published a book ``L'int\'{e}gration dans les groupes topologique'', in which he demonstrated that Theorem \ref{KolmogorovRiesz} holds true even if we replace $\reals^N$ with a locally compact Hausdorff group $G$\footnote{See \cite{Weil}, p. 53-54 for Weil's argument (in french).}:

\begin{thm}
Let $G$ be a locally compact Hausdorff group. A family $\Ffamily\subset L^p(G)$ is relatively compact if and only if
\begin{itemize}
	\item $\Ffamily$ is $L^p-$bounded,
	\item $\Ffamily$ is $L^p-$equicontinuous, i.e. for every $\eps>0$ there exists an open neighbourhood $U_e$ of the neutral element $e$ such that 
	$$\forall_{\substack{y\in U_e \\ f\in\Ffamily}}\ \int_G\ |f(xy) - f(x)|^p\ dx \leqslant \eps,$$
	\item $\Ffamily$ is $L^p-$equivanishing, i.e. for every $\eps>0$ there exists $K\Subset G$ (which means that $K$ is a compact subset of $G$) such that 
	$$\forall_{f\in\Ffamily}\ \int_{G\backslash K}\ |f(x)|^p \ dx \leqslant \eps.$$
\end{itemize}
\end{thm}

Weil's book is oftentimes cited as a reference source when it comes to characterizing relatively compact families in $L^p-$spaces\footnote{See \cite{Feichtinger, GorkaKostrzewa, GorkaMacios, GorkaRafeiro, HancheOlsenHolden}.}. However, the argument of the french mathematician is (at least in the author's opinion) rather difficult to follow $-$ the exposition is very terse, avoids technical details and leaves much of the work to the reader. The fact that the book is written in French does not make matters easier. Our goal is to overcome these inconveniences. 

As far as the organization of the paper is concerned, Chapter \ref{sectioncompactfamiliesC0} is devoted to the Arzel\`a-Ascoli theorem. As we document, the classic version of this theorem (for $C(X),\ X$ being a compact space) is very well-established in the literature, but the version for $C_0(X),\ X$ being a locally compact space, is almost non-existent. In our approach we employ the concept of one-point compactification $X_{\infty}$ of $X$. Following a brief description of this construct we lay out a full proof of the Arzel\`a-Ascoli theorem for $C_0(X).$

Chapter \ref{sectioncompactfamiliesLp} commences with a short summary of the Haar measure. Subsequently, we introduce three $L^p-$properties: $L^p-$boundedness, $L^p-$equicontinuity and $L^p-$equivanishing. Next, Theorems \ref{Lpboundedness}, \ref{Lpequicont} and \ref{Lpequivanish} demonstrate how these $L^p-$properties are ``passed down'' from $\Ffamily\subset L^p(G), G$ being a locally compact group, to $\Ffamily\star \phi\subset C_0(G),$ where $\phi\in C_c(G).$

Chapter \ref{sectionYoung} is divided into two parts: the first one pivots around proving that $\Ffamily\star\phi$ is ``not far away'' (if we choose $\phi\in C_c(G)$ properly) from $\Ffamily$ (see Theorem \ref{choiceofphi}). The subject of the second part is Young's convolution inequality, which we generalize beyond what is currently known in the literature. 

Finally, Chapter \ref{sectionmainpart} brings all the ``pieces of the puzzle'' from previous chapters together. In the climax of the paper (see Theorem \ref{KolmogorovRieszWeilSudakovtheorem}) we provide an elegant and natural proof of the characterization of the relatively compact families in $L^p(G).$

\section{Arzel\`a-Ascoli theorem and one-point compactification}
\label{sectioncompactfamiliesC0}

The classic version of the Arzel\`a-Ascoli theorem characterizes relatively compact families in $C(X),$ the space of continuous functions on a compact space $X$ (we assume that \textit{every} topological space in this paper is Hausdorff). Such version of the result appears in many sources throughout the literature\footnote{See Theorem 4.43 in \cite{Follandrealanalysis}, p. 137 or Corollary 10.49 in \cite{Knapp}, p. 479 or Theorem A5 in \cite{Rudin}, p. 394.}, although some authors assume additionally that the space $X$ is metric for ``simplicity of the proof''\footnote{See Theorem 4.25 in \cite{Brezis}, p. 111 or Theorem 23.2 in \cite{Choquet}, p. 99 or Theorem 6.3.1 in \cite{Dixmiertopology}, p. 69 or Theorem 11.28 in \cite{Rudinrealandcomplex}, p. 245.}.

\begin{thm}(classic version of the Arzel\`a-Ascoli theorem)\\
Let $X$ be a compact space. The family $\Ffamily\subset C(X)$ is relatively compact if and only if 
\begin{itemize}
	\item $\Ffamily$ is pointwise bounded, i.e. for every $x\in X$ there exists $M_x>0$ such that 
	$$\forall_{f\in\Ffamily}\ |f(x)| \leqslant M_x,$$
	\item $\Ffamily$ is equicontinuous at every point, i.e. for every $x\in X$ and $\eps>0$ there exists an open neighbourhood $U_x$ of $x$ such that 
	$$\forall_{\substack{y\in U_x,\\ f\in\Ffamily}}\ |f(y)- f(x)| \leqslant \eps.$$
\end{itemize}
\label{classicalAA}
\end{thm}

For our purposes, we need a more general result. We need a characterization of relatively compact families in $C_0(X),$ the space of continuous functions on a locally compact space $X$ which vanish at infinity, i.e. for every $\eps>0$ there exists a compact set $K$ in $X$ (we denote this situation by $K\Subset X$) such that 
$$\forall_{x\in X\backslash K}\ |f(x)| \leqslant \eps.$$ 

\noindent
Looking through the literature, one usually stumbles across versions of the Arzel\`a-Ascoli theorem in which the supremum-norm topology is replaced with the compact-open topology\footnote{See Theorem 3.4.20 in \cite{Engelking}, p. 163 or Theorem 17 in \cite{Kelley}, p. 233 or Theorem 47.1 in \cite{Munkres}, p. 290.}. 

\begin{thm}(Arzel\`a-Ascoli theorem for $C_0(X)$ with the compact-open topology)\\
Let $X$ be a locally compact space. The family $\Ffamily \subset C_0(X)$ is relatively compact in the compact-open topology if and only if
\begin{itemize}
	\item $\Ffamily$ is pointwise bounded,
	\item $\Ffamily$ is equicontinuous at every point.
\end{itemize}
\label{AAC0compactopentopology}
\end{thm}

The above result seems promising (to say the least), but it has a ``splinter'' in the form of the switch of topologies. Is there a version of the Arzel\`a-Ascoli theorem for $C_0(X)$ in which the supremum-norm topology is not replaced with compact-open topology (or any other topology for that matter)? The literature is surprisingly scarce in this respect. One possible reference is Exercise 17 on page 182 in John B. Conway's ``A Course in Functional Analysis''\footnote{See \cite{Conway}.}. However, Conway does not provide a solution to his exercise, leaving this task to the reader. Another reference is the monograph ``Integral Systems and Stability of Feedback Systems'' by Constantin Corduneanu\footnote{See \cite{Corduneanu}, p. 62.}. Unlike Conway, Corduneanu does provide a proof, but (in the author's opinion) he does not provide an insight into \textit{why} the theorem is true\footnote{The idea of Cordeanu's proof is basically to take a sequence and repeatedly choose some subsequences (using the assumptions) until we arrive at a convergent subsequence. This demonstrates the relative compactness of the family in question but still leaves the reader wondering ``why is this theorem true?''.}. We aim to rectify this imperfection.

For any locally compact space $X$, there exists a compact space $X_{\infty}$ such that
\begin{itemize}
	\item $X$ is a subspace of $X_{\infty},$
	\item the closure of $X$ is $X_{\infty},$
	\item $X_{\infty}\backslash X$ is a singleton or an empty set.
\end{itemize}
 
\noindent
The space $X_{\infty}$ is called the \textit{one-point compactification} of $X$\footnote{See Theorem 3.5.11 in \cite{Engelking}, p. 169 or Theorem 29.1 in \cite{Munkres}, p. 183 or Proposition 1.7.3 in \cite{Pedersen}, p. 37.}. Its topology allows for a simple description $-$ it consists of open sets in $X$ (the topology of $X$) plus all sets of the form $X_{\infty}\backslash K,$ where $K\Subset X$. 

If $X$ happens to be compact itself, then $X = X_{\infty}.$ Otherwise, $X_{\infty}\backslash X$ is a singleton, whose element is often denoted by $\infty$. A common example of the one-point compactification is $\complex_{\infty},$ which arises in the field of complex analysis as a natural domain for M\"obius transformations. It turns out that $\complex_{\infty}$ is homeomorphic to a sphere $S^2,$ called the \textit{Riemann sphere}, and the element $\infty$ may be regarded as the ``north pole'' of that sphere\footnote{See \cite{GilmanKraRodriguez}, p. 38 or \cite{GongGong}, p. 10 or \cite{SteinShakarchi}, p. 88-89.}.

\begin{thm}(Arzel\`a-Ascoli theorem for $C_0(X)$ with the supremum-norm topology)\\
Let $X$ be a locally compact space. The family $\Ffamily \subset C_0(X)$ is relatively compact in the supremum-norm topology if and only if
\begin{itemize}
	\item $\Ffamily$ is pointwise bounded,
	\item $\Ffamily$ is equicontinuous at every point,
	\item $\Ffamily$ is equivanishing, i.e. for every $\eps>0$ there exists $K\Subset X$ such that 
	\begin{gather} 
	\forall_{\substack{x\not\in K\\ f\in\Ffamily}}\ |f(x)| \leqslant \eps.
	\label{equivanishing}
	\end{gather}
\end{itemize}
\label{AAC0}
\end{thm}
\begin{proof}
For every $f\in C_0(X)$ we define a function $\Psi_f:X_{\infty}\longrightarrow \complex$ such that $\Psi_f|_X = f$ and $\Psi_f(\infty) = 0.$ Obviously, every $\Psi_f$ is continuous at each element $x\in X$. Furthermore, functions $\Psi_f$ are also continuous at $\infty:$ for a fixed $\eps>0$ we choose $K\Subset X$ such that (\ref{equivanishing}) is satisfied and put $U_{\infty} := X_{\infty}\backslash K$, which is an open neighbourhood of $\infty.$ Consequently, (\ref{equivanishing}) corresponds to 
$$\forall_{x\in U_{\infty}}\ |\Psi_f(x) - \Psi_f(\infty)|\leqslant \eps,$$

\noindent
which demonstrates that $\Psi_f \in C(X_{\infty})$ for every $f\in C_0(X).$

The crucial observation is that for every $f\in C_0(X)$ we have
$$\sup_{x\in X}\ |f(x)| = \sup_{x\in X_{\infty}}\ |\Psi_f(x)|,$$

\noindent
so the function $\Psi: C_0(X) \longrightarrow C(X_{\infty}),$ given by $\Psi(f) := \Psi_f,$ is an isometry. It is now clear that the following conditions are equivalent: 
\begin{itemize}
	\item $\Ffamily$ is relatively compact in $C_0(X),$
	\item $\Psi(\Ffamily)$ is relatively compact in $C(X_{\infty}),$
	\item $\Psi(\Ffamily)$ is pointwise bounded and equicontinuous at every point $x\in X_{\infty}$ (by the classic Arzel\`a-Ascoli theorem).
\end{itemize}
 
\noindent
Pointwise boundedness of $\Psi(\Ffamily)$ is equivalent to pointwise boundedness of $\Ffamily$ (since for every $\Psi_f\in\Psi(\Ffamily)$ we have $\Psi_f(\infty) = 0$). Furthermore, equicontinuity of $\Psi(\Ffamily)$ at element $x\in X$ is equivalent to equicontinuity of $\Ffamily$ at this element. Last but not least, the equicontinuity of $\Psi(\Ffamily)$ at $\infty$ means that for every $\eps>0$ there exists $K\Subset X$ such that 
$$\forall_{\substack{x\in X_{\infty}\backslash K\\ F\in\Psi(\Ffamily)}}\ |F(x) - F(\infty)| \leqslant \eps,$$

\noindent
which is equivalent to equivanishing of $\Ffamily.$ This concludes the proof.
\end{proof}

\section{$L^p-$properties}
\label{sectioncompactfamiliesLp}

Throughout the whole paper, we work under the assumption that
\begin{center}
\textit{$G$ is a locally compact (Hausdorff) group.}
\end{center}

\noindent
It turns out\footnote{See Chapter 1.3 in \cite{DeitmarEchterhoff} or Chapter 2.2 in \cite{FollandAHA} or Section 15 in \cite{HewittRoss} or Chapter 2 in \cite{Weil}.} that any such group $G$ admits a \textit{Haar measure} $\mu,$ i.e. a nonzero, Borel measure which is 
\begin{itemize}
	\item finite on compact sets,
	\item \textit{inner regular}, i.e. for every open set $U$ we have
	$$\mu(U) = \sup\ \{\mu(K)\ :\ K\subset U,\ K - \text{ compact}\},$$
	
	\item \textit{outer regular}, i.e. for every measurable set $A$ we have
	$$\mu(A) = \inf\ \{\mu(U)\ :\ A \subset U,\ U - \text{ open}\},$$
	
	\item \textit{left-invariant}, i.e. for every $x\in G$ and measurable set $A$ we have $\mu(xA) = \mu(A).$
\end{itemize}

\noindent
Haar measure is not determined uniquely, but is ``unique up to a positive constant'', i.e. if $\mu_1,\mu_2$ are two (left) Haar measures on $G$ then there exists a constant $c>0$ such that $\mu_1 = c\cdot \mu_2$. Furthermore, similarly to the \textit{left} Haar measure, one can also prove the existence (and uniqueness up to positive constant) of the \textit{right} Haar measure. In general, these two objects need not coincide on $G$ - if they do, we call $G$ \textit{unimodular}. The family of unimodular groups contains all locally compact \textit{abelian} groups as well as all \textit{compact} groups.

As far as the history is concerned, the construction of the Haar measure is commonly attributed to Alfr\'ed Haar, although the contribution of Henri Cartan or Andr\'e Weil deserves recognition as well\footnote{See \cite{Cartan, Haar} or Chapter 2 in \cite{Weil}.}. For a detailed account of the subject we refer the Reader to Diestel's and Spalsbury's monograph ``The joys of Haar measure''\footnote{See \cite{DiestelSpalsbury}.}.

From this point onwards
\begin{center}
\textit{we assume that $p\geqslant 1$ and $p'$ is its H\"{o}lder conjugate.}
\end{center} 

\noindent
Our aim is to characterize relatively compact families of $L^p(G)$ and to achieve this goal we introduce three concepts: firstly, we say that a family $\Ffamily\subset L^p(G)$ is \textit{$L^p-$bounded} if there exists $M>0$ such that 
$$\forall_{f\in\Ffamily}\ \|f\|_p \leqslant M.$$

\noindent
Secondly, a family $\Ffamily\subset L^p(G)$ is called \textit{$L^p-$equicontinuous} if for every $\eps>0$ there exists an open neighbourhood $U_e$ of the neutral element $e\in G$ such that  
$$\forall_{f\in\Ffamily}\ \sup_{x\in U_e}\ \|L_xf-f\|_p \leqslant \eps \hspace{0.4cm}\text{and}\hspace{0.4cm} \sup_{x\in U_e}\ \|R_xf-f\|_p\leqslant \eps,$$

\noindent
where $L_xf(y) := f(xy)$ and $R_xf(y) := f(yx)$ are the \textit{left} and \textit{right shift operators}, respectively. Thirdly, $\Ffamily\subset L^p(G)$ is said to be \textit{$L^p-$equivanishing} if for every $\eps>0$ there exists $K\Subset G$ such that  
$$\forall_{f\in \Ffamily}\ \int_{G\backslash K}\ |f(y)|^p\ dy \leqslant \eps.$$

We now demonstrate that all three properties above ($L^p-$boundedness, $L^p-$equicontinuity and $L^p-$equivanishing) are ``inherited'' when the family $\Ffamily\subset L^p(G)$ is convolved with a function $\phi\in C_c(G),$ i.e. a continuous function with compact support. 

As far as the notation is concerned, $C^b(G)$ in the theorem below denotes the Banach space (with supremum-norm) of continuous and bounded functions on $G$.

\begin{thm}
Let $\Ffamily\subset L^p(G)$ be $L^p-$bounded. If $\phi \in C_c(G)$ then $\Ffamily\star \phi\subset C^b(G)$ is bounded. 
\label{Lpboundedness}
\end{thm}
\begin{proof}
Let $M>0$ be a $L^p-$bound on the family $\Ffamily$. We divide the proof into two steps.

\vspace{0.3cm}
\noindent
\textbf{Step 1.} \textit{Proving that $f\star \phi$ is continuous for every $f\in \Ffamily$}

\vspace{0.3cm}
Fix $\eps>0,\ x_*\in G$ and suppose that $p > 1$. By Proposition 2.41 in \cite{FollandAHA}, p. 53 there exists a symmetric open neighbourhood $U_e$ of the neutral element such that 
\begin{gather}
\forall_{x\in U_e}\ \left(\int_G\ |\phi\circ\iota(xy) - \phi\circ\iota(y)|^{p'}\ dy\right)^{\frac{1}{p'}} \leqslant \frac{\eps}{M},
\label{continuityepsoverM}
\end{gather}

\noindent
where $\iota:G\longrightarrow G$ stands for the \textit{inverse function}, i.e. $\iota(x) := x^{-1}$. For $x\in x_*U_e$ and $f\in\Ffamily$ we have 
\begin{equation*}
\begin{split}
|f\star\phi(x) - f\star\phi(x_*)| &= \left|\int_G\ f(y)\cdot \phi\left(y^{-1}x\right) - f(y)\cdot \phi\left(y^{-1}x_*\right)\ dy\right| \\
\stackrel{\text{H\"{o}lder ineq.}}{\leqslant}& \|f\|_p\cdot \left(\int_G\ \left|\phi\left(y^{-1}x\right) - \phi\left(y^{-1}x_*\right)\right|^{p'}\ dy \right)^{\frac{1}{p'}} \\
\stackrel{y\mapsto x_*y}{\leqslant}& M\cdot \left(\int_G\ \left|\phi\left(y^{-1}x_*^{-1}x\right) - \phi\left(y^{-1}\right)\right|^{p'}\ dy \right)^{\frac{1}{p'}}
\\
=& M\cdot \left(\int_G\ |\phi\circ\iota\left(x^{-1}x_*y\right) - \phi\circ\iota\left(y\right)|^{p'}\ dy \right)^{\frac{1}{p'}} \stackrel{(\ref{continuityepsoverM})}{\leqslant} \eps.
\end{split}
\end{equation*}

\noindent
This proves that $\Ffamily\star\phi$ is a family of continuous functions if $p > 1$. 

For $p = 1$ let us again fix $\eps > 0$ and $x_*\in G.$ By Lemma 1.3.6 in \cite{DeitmarEchterhoff}, p. 11 function $\phi$ is uniformly continuous, so there exists a symmetric open neighbourhood $V_e$ of the neutral element such that 
\begin{gather}
\forall_{x\in x_*V_e}\sup_{y\in G}\ \left|\phi\left(y^{-1}x\right) - \phi\left(y^{-1}x_*\right)\right| \leqslant \frac{\eps}{M}.
\label{phiisunicont} 
\end{gather}

\noindent
Consequently, for $x\in x_*V_e$ and $f\in\Ffamily$ we have
\begin{gather*}
|f\star\phi(x) - f\star\phi(x_*)| = \left|\int_G\ f(y)\cdot \phi\left(y^{-1}x\right) - f(y)\cdot \phi\left(y^{-1}x_*\right)\ dy\right| \\
\leqslant \sup_{y\in G}\ \left|\phi\left(y^{-1}x\right) - \phi\left(y^{-1}x_*\right)\right| \cdot \int_G\ |f(y)|\ dy
\stackrel{(\ref{phiisunicont})}{\leqslant} \frac{\eps}{M}\cdot \|f\|_1 \leqslant \eps.
\end{gather*}

\noindent
This demonstrates that $\Ffamily\star\phi$ is a family of continuous functions if $p=1.$

\vspace{0.3cm}
\noindent
\textbf{Step 2.} \textit{Proving that the family $\Ffamily\star \phi$ is bounded in supremum norm}

\vspace{0.3cm}
Suppose that $p>1$ and put $K := \iota(\supp(\phi)).$ We observe that
\begin{gather*}
\forall_{\substack{x\in G\\ f\in\Ffamily}}\ \bigg|\int_G\ f(y)\cdot \phi\left(y^{-1}x\right)\ dy\bigg| \stackrel{\text{H\"{o}lder ineq.}}{\leqslant} \|f\|_p\cdot \left(\int_G\ \left|\phi\left(y^{-1}x\right)\right|^{p'}\ dy\right)^{\frac{1}{p'}} \\
\leqslant M\cdot \|\phi\|_{\infty}\cdot \mu\big(xK\big)^{\frac{1}{p'}} = M\cdot \|\phi\|_{\infty}\cdot \mu(K)^{\frac{1}{p'}},
\end{gather*}

\noindent
where the second inequality stems from the fact that if 
$$y^{-1}x \not\in\supp(\phi) \ \Longleftrightarrow \ y\not\in xK,$$

\noindent
then $\phi\big(y^{-1}x\big) = 0$. We conclude that $\Ffamily\star\phi\subset C^b(G)$ is bounded if $p > 1$.

For $p = 1$, the reasoning is even simpler:
\begin{gather*}
\forall_{\substack{x\in G\\ f\in\Ffamily}}\ \bigg|\int_G\ f(y)\cdot \phi\left(y^{-1}x\right)\ dy\bigg| \leqslant \int_G\ |f(y)|\ dy\cdot \|\phi\|_{\infty} \leqslant M\cdot \|\phi\|_{\infty}.
\end{gather*}

\noindent
This concludes the proof.
\end{proof}

Going off on a tangent for a brief moment, one may wonder whether $\mu(\iota(\supp(\phi)))$ (which appears in the second step of the proof above) is equal to $\mu(\supp(\phi)).$ In general, this is not the case! However if (and only if) the group $G$ is unimodular (left and right Haar measure coincide) then for every measurable set $A$ we have $\mu(A) = \mu(\iota(A))$\footnote{See \cite{Pap}, p. 1112.}.

\begin{thm}
Let $\Ffamily\subset L^p(G)$ be $L^p-$equicontinuous. If $\phi\in C_c(G)$, then $\Ffamily\star\phi$ is equicontinuous. 
\label{Lpequicont}
\end{thm}
\begin{proof}
Firstly, suppose that $p > 1,\ \eps>0,\ x_*\in G$ and $K := \iota(\supp(\phi))$. By $L^p-$equicontinuity of the family $\Ffamily$, let $U_e$ be a symmetric open neighbourhood of the neutral element such that 
\begin{gather}
\forall_{f\in\Ffamily}\ \sup_{x\in U_e}\ \|L_xf-f\|_p \leqslant \frac{\eps}{\|\phi\|_{\infty}\cdot \mu(K)^{\frac{1}{p'}}}.
\label{Lxfminusfpfraceps}
\end{gather}

\noindent
Observe that for every $x\in G$ and $f\in\Ffamily$ we have 
\begin{gather}
f\star\phi(x) = \int_G\ f(y)\cdot \phi\left(y^{-1}x\right)\ dy \stackrel{y\mapsto xy}{=} \int_G\ f(xy)\cdot \phi\left(y^{-1}\right)\ dy.
\label{fstarphi} 
\end{gather}

\noindent
Consequently, for every $x\in x_*U_e$ and $f\in\Ffamily$ we have
\begin{gather*}
|f\star\phi(x) - f\star\phi(x_*)| \stackrel{(\ref{fstarphi})}{\leqslant} \int_G\ |f(xy)-f(x_*y)|\cdot \left|\phi\left(y^{-1}\right)\right|\ dy \\
\stackrel{\text{H\"{o}lder ineq.}}{\leqslant} \left(\int_G\ |f(xy)-f(x_*y)|^p\ dy\right)^{\frac{1}{p}}\cdot \left(\int_G\ |\phi\left(y^{-1}\right)|^{p'}\ dy\right)^{\frac{1}{p'}}\\
\leqslant \left(\int_G\ |f(xy)-f(x_*y)|^p\ dy\right)^{\frac{1}{p}}\cdot \|\phi\|_{\infty}\cdot \mu(K)^{\frac{1}{p'}} \\
\stackrel{y\mapsto x_*^{-1}y}{=} \left(\int_G\ \left|f\left(xx_*^{-1}y\right)-f(y)\right|^p\ dy\right)^{\frac{1}{p}}\cdot \|\phi\|_{\infty}\cdot \mu(K)^{\frac{1}{p'}} \stackrel{(\ref{Lxfminusfpfraceps})}{\leqslant}\eps,
\end{gather*}

\noindent
which ends the proof if $p>1$.

For $p=1$ let us again fix $\eps>0$ and $x_*\in G$. By $L^1-$equicontinuity of the family $\Ffamily,$ let $U_e$ be a symmetric open neighbourhood of the neutral element such that 
\begin{gather}
\forall_{f\in\Ffamily}\ \sup_{x\in U_e}\ \|L_xf - f\|_1 \leqslant \frac{\eps}{\|\phi\|_{\infty}}.
\label{whatifweuseL1equicont}
\end{gather}

\noindent
For every $x\in x_*U_e$ and $f\in\Ffamily$, we obtain
\begin{gather*}
|f\star\phi(x) - f\star\phi(x_*)| \stackrel{(\ref{fstarphi})}{\leqslant} \int_G\ |f(xy)-f(x_*y)|\cdot \left|\phi\left(y^{-1}\right)\right|\ dy \\
\leqslant \|\phi\|_{\infty}\cdot \int_G\ |f(xy)-f(x_*y)|\ dy \stackrel{y\mapsto x_*^{-1}y}{=} \|\phi\|_{\infty}\cdot \int_G\ \left|f\left(xx_*^{-1}y\right)-f(y)\right|\ dy \stackrel{(\ref{whatifweuseL1equicont})}{\leqslant}\eps,
\end{gather*}

\noindent
which ends the proof.
\end{proof}

Before we demonstrate how $L^p-$equivanishing behaves ``under convolution'', we note the following result:

\begin{lemma}
If $K_1$ and $K_2$ are compact subsets of $G$, then there exists a compact set $D$ such that 
\begin{gather}
\forall_{x\not\in D}\ xK_1\cap K_2 = \emptyset.
\label{Dproperty}
\end{gather}
\label{constructionofD}
\end{lemma}
\begin{proof}
We put $D:= K_2\cdot K_1^{-1}.$ It is obviously compact by the continuity of group operations and the compactness of $K_1$ and $K_2.$ To prove that (\ref{Dproperty}) is true, suppose that $z\in xK_1\cap K_2$ for some $x\not\in D.$ In particular $z \in K_2$ and $z = xy$ for some $y\in K_1.$ This means that 
$$x = zy^{-1} \in K_2\cdot K_1^{-1} = D,$$

\noindent
which is a contradiction. This concludes the proof.
\end{proof}

\begin{thm}
Let $\Ffamily\subset L^p(G)$ be $L^p-$equivanishing. If $\phi\in C_c(G)$ is such that $\phi(e)\neq 0$ ($e$ is the neutral element of $G$), then $\Ffamily\star\phi$ is equivanishing. 
\label{Lpequivanish}
\end{thm}
\begin{proof}
Firstly, suppose that $p > 1$. We fix $\eps>0$ and choose $K\Subset G$ such that 
\begin{gather}
\forall_{f\in\Ffamily}\ \left(\int_{G\backslash K}\ |f|^p\ d\mu\right)^{\frac{1}{p}} \leqslant \frac{\eps}{\|\phi\|_{\infty}\cdot \mu\big(\iota(\supp(\phi))\big)^{\frac{1}{p'}}}.
\label{GbackslashKepsilon}
\end{gather}

\noindent
Let us denote $U := \iota(\{\phi\neq0\})$, which is open and relatively compact neighbourhood of the neutral element with $\overline{U} = \iota(\supp(\phi))$. By Lemma \ref{constructionofD}, there exists $D\Subset G$ such that 
\begin{gather}
\forall_{x\not\in D}\ x\overline{U} \cap K = \emptyset.
\label{capKemptyset}
\end{gather}

\noindent
For $x\not\in D$ and $f\in\Ffamily$ we have
\begin{gather*}
|f\star\phi(x)| \leqslant \int_G\ \left|f(y)\cdot \phi\left(y^{-1}x\right)\right|\ dy \leqslant \|\phi\|_{\infty}\cdot \int_{x\overline{U}}\ |f|\ d\mu \\
\stackrel{\text{H\"{o}lder ineq.}}{\leqslant} \|\phi\|_{\infty}\cdot \left(\int_{x\overline{U}}\ |f|^p\ d\mu\right)^{\frac{1}{p}}\cdot \mu(x\overline{U})^{\frac{1}{p'}} \\
\stackrel{(\ref{capKemptyset})}{\leqslant} \|\phi\|_{\infty}\cdot \left(\int_{G\backslash K}\ |f|^p\ d\mu\right)^{\frac{1}{p}}\cdot \mu\big(x\overline{U}\big)^{\frac{1}{p'}} \stackrel{(\ref{GbackslashKepsilon})}{\leqslant} \eps,
\end{gather*}

\noindent
which ends the proof if $p > 1$.

For $p=1$ let us again fix $\eps>0$ and choose $K\Subset G$ such that 
\begin{gather}
\forall_{f\in\Ffamily}\ \int_{G\backslash K}\ |f|\ d\mu \leqslant \frac{\eps}{\|\phi\|_{\infty}}.
\label{GbackslashKepsilon2}
\end{gather}

\noindent
Furthermore, we choose $U$ and $D$ as previously. For every $f\in\Ffamily$ and $x\not\in D$ we have
\begin{gather*}
|f\star\phi(x)| \leqslant \int_G\ \left|f(y)\cdot \phi\left(y^{-1}x\right)\right|\ dy \leqslant \|\phi\|_{\infty}\cdot \int_{x\overline{U}}\ |f|\ d\mu \stackrel{(\ref{capKemptyset})}{\leqslant} \|\phi\|_{\infty}\cdot \int_{G\backslash K}\ |f|\ d\mu \stackrel{(\ref{GbackslashKepsilon2})}{\leqslant} \eps.
\end{gather*}

\noindent
This concludes the proof.
\end{proof}

\section{Approximation theorem and Young's convolution inequality}
\label{sectionYoung}

The first part of this chapter is devoted to demonstrating that the family $\Ffamily\star\phi$ (for a suitably chosen $\phi\in C_c(G)$) approximates the family $\Ffamily$ in $L^p-$norm. To this end, we recall \textit{Minkowski's integral inequality}\footnote{See Theorem 6.19 in \cite{Follandrealanalysis}, p. 194.}:

\begin{thm}
Let $X,Y$ be $\sigma-$finite measure spaces, $1\leqslant p < \infty$ and let $F:X\times Y \longrightarrow \complex$ be a measurable function. Then
\begin{gather}
\left(\int_X \left(\int_Y\ |F(x,y)|\ dy\right)^p dx\right)^{\frac{1}{p}} \leqslant \int_Y \left(\int_X\ |F(x,y)|^p\ dx\right)^{\frac{1}{p}} dy.
\label{Minkineq}
\end{gather}
\end{thm}

To a certain degree, the theorem below resembles Proposition 2.42 in Folland's ``A Course in Abstract Harmonic Analysis''\footnote{See \cite{FollandAHA}, p. 53.}. However, Folland's proposition focuses on a single function $f$ and our approximation theorem deals with an $L^p-$equicontinuous family $\Ffamily\subset L^p(G)$.

\begin{thm}
If $\Ffamily\subset L^p(G)$ is $L^p-$equicontinuous, then for every $\eps>0$ there exists a function $\phi \in C_c(G)$ such that 
$$\forall_{f\in \Ffamily}\ \|f\star \phi - f\|_p \leqslant \eps.$$
\label{choiceofphi}
\end{thm}
\begin{proof}
Fix $\eps>0$ and let $U_e$ be the open neighbourhood of the neutral element such that 
\begin{gather}
\forall_{f\in \Ffamily}\ \sup_{y\in U_e}\ \|R_yf - f\|_p \leqslant \eps.
\label{Ryfminusfpeps}
\end{gather}

\noindent
We fix $f\in\Ffamily$ and choose $f_B$ to be a Borel-measurable function such that $f = f_B$ almost everywhere. Since $G$ is a Tychonoff space, we can pick $\phi\in C_c(G)$ such that 
\begin{itemize}
	\item $\phi(e)\neq 0,\ \phi \geqslant 0,$
	\item $\int_G\ \phi\circ \iota\ d\mu = 1,$
	\item $\supp(\phi\circ\iota)\subset U_e.$
\end{itemize}

\noindent
For every $x\in G$ we have 
\begin{equation*}
\begin{split}
f\star\phi(x)-f(x) &= \int_G\ f(y)\cdot \phi\left(y^{-1}x\right)\ dy - f(x)\cdot \int_G\ \phi\left(y^{-1}\right)\ dy \\
&= \int_G\ \big(f(xy) - f(x)\big)\cdot \phi\left(y^{-1}\right)\ dy \\
&= \int_G\ \big(f_B(xy) - f_B(x)\big)\cdot \phi\left(y^{-1}\right)\ dy.
\end{split}
\end{equation*}

We put 
$$F(x,y) := \big(f_B(xy) - f_B(x)\big)\cdot \phi\left(y^{-1}\right),$$ 

\noindent
which is Borel-measurable as a composition of the following Borel-measurable functions:
\begin{equation*}
\begin{split}
&F_1 : (x,y)\mapsto (x,y,y),\\
&F_2 : (x,y,z)\mapsto (x,y,z^{-1}),\\
&F_3 : (x,y,z)\mapsto (x,xy,\phi(z)),\\
&F_4 : (x,y,z)\mapsto \left(f_B(x),f_B(y),z\right),\\
&F_5 : (x,y,z)\mapsto (y-x)z. 
\end{split}
\end{equation*}

\noindent
Observe that if we replace $f_B$ with $f$ in $F_4$, then $F$ does not need to be Borel-measurable (or even measurable), since a composition of a measurable function with a continuous function need not be measurable.
 
Furthermore, since $f_B$ and $\phi$ are integrable with $p-$th power, then $\supp(f_B)$ and $\supp(\phi)$ are $\sigma-$compact\footnote{See Corollary 1.3.5 in \cite{DeitmarEchterhoff}, p. 10.}. Hence, also the sets 
$$\left(\supp(f_B)\cdot\supp(\phi)\right)\times\iota(\supp(\phi)) \hspace{0.4cm}\text{and}\hspace{0.4cm} \supp(f_B)\times \iota(\supp(\phi))$$

\noindent
are $\sigma-$compact. Following a series of logical implications:
\begin{gather*}
(x,y)\in \{F\neq 0\} \ \Longrightarrow \ \bigg(f_B(xy)-f_B(x)\neq 0 \hspace{0.4cm}\text{and}\hspace{0.4cm} \phi\left(y^{-1}\right) \neq 0\bigg)\\
\Longrightarrow\ \bigg(\left(xy\in\supp(f_B) \hspace{0.4cm}\text{or}\hspace{0.4cm} x\in\supp(f_B)\right) \hspace{0.4cm}\text{and}\hspace{0.4cm} y^{-1}\in\supp(\phi) \bigg)\\
\Longrightarrow\ \bigg(\left(xy\in\supp(f_B) \hspace{0.4cm}\text{and}\hspace{0.4cm} y^{-1}\in\supp(\phi)\right) \hspace{0.4cm}\text{or}\hspace{0.4cm} \left(x\in\supp(f_B) \hspace{0.4cm}\text{and} \hspace{0.4cm} y^{-1}\in\supp(\phi)\right) \bigg)\\
\Longrightarrow\ \bigg((x,y)\in \left(\supp(f_B)\cdot\supp(\phi)\right)\times\iota(\supp(\phi)) \hspace{0.4cm}\text{or}\hspace{0.4cm} (x,y)\in\supp(f_B)\times\iota(\supp(\phi))\bigg)
\end{gather*}

\noindent
we conclude that $\{F\neq 0\}$ is $\sigma-$compact. 

Finally, we are in position to apply Minkowski's integral inequality:
\begin{equation*}
\begin{split}
\forall_{f\in\Ffamily}\ \|f\star\phi - f\|_p &= \|f_B\star\phi - f_B\|_p = \left(\int_G\ \bigg|\int_G\ \big(f_B(xy)-f_B(x)\big)\cdot \phi\left(y^{-1}\right)\ dy\bigg|^p\ dx\right)^{\frac{1}{p}} \\
&\stackrel{(\ref{Minkineq})}{\leqslant} \int_G\ \left( \int_G\ |f_B(xy)-f_B(x)|^p\cdot \left|\phi\left(y^{-1}\right)\right|^p\ dx\right)^{\frac{1}{p}} \ dy \\
&= \int_G\ \|R_yf-f\|_p\cdot \left|\phi\left(y^{-1}\right)\right| \ dy \leqslant \sup_{y\in U}\ \|R_yf-f\|_p \stackrel{(\ref{Ryfminusfpeps})}{\leqslant} \eps.
\end{split}
\end{equation*}

\noindent
This concludes the proof.
\end{proof}

We now turn to Young's convolution inequality. Let $\Delta$  represent the \textit{modular function} associated with the Haar measure $\mu$\footnote{See Chapter 1.4 in \cite{DeitmarEchterhoff} or Chapter 2.4 in \cite{FollandAHA} or Section 15 in \cite{HewittRoss} or Chapter 3.3 in \cite{ReiterStegeman}.}. In \cite{QuekYap} Quek and Yap prove Young's convolution inequality for unimodular groups $-$ these are locally compact groups in which $\Delta = 1$ (equivalently, left and right Haar measures coincide). Below we present our take on Young's convolution inequality without the assumption of unimodularity.

\begin{thm}
Let $p,q,r \in [1,\infty)$ be such that 
\begin{gather}
\frac{1}{r} = \frac{1}{p} + \frac{1}{q} - 1.
\label{pqrequality}
\end{gather}

\noindent
For functions $f\in L^p(G)$ and $g\in L^q(G)$, the convolution $f\star \left(\Delta^{\frac{1}{p'}}g\right)$ exists almost everywhere. Moreover, $f\star \left(\Delta^{\frac{1}{p'}}g\right) \in L^r(G)$ and we have
\begin{gather}
\|f\star \left(\Delta^{\frac{1}{p'}}g\right)\|_r \leqslant \|f\|_p \cdot \|g\|_q.
\label{younginequality}
\end{gather}
\label{Youngtheorem}
\end{thm}
\begin{proof}
Observe that it suffices to prove (\ref{younginequality}), which will immediately establish that $f\star \left(\Delta^{\frac{1}{p'}}g\right) \in L^r(G)$ and consequently that the convolution exists almost everywhere.  

Firstly, we note a couple of useful equalities:
\begin{equation}
\begin{split}
&\frac{1}{r} + \frac{1}{q'} + \frac{1}{p'} = \frac{1}{r} + \bigg(1-\frac{1}{q}\bigg) + \bigg(1-\frac{1}{p}\bigg) \stackrel{(\ref{pqrequality})}{=} 1,\\
&\bigg(1-\frac{p}{r}\bigg)q' \stackrel{(\ref{pqrequality})}{=} p\bigg(1-\frac{1}{q}\bigg)q' = p, \\
&\bigg(1-\frac{q}{r}\bigg)p' \stackrel{(\ref{pqrequality})}{=} q\bigg(1-\frac{1}{p}\bigg)p' = q.
\end{split}
\label{coupleofequalities}
\end{equation}

\noindent
With the aid of H\"{o}lder's inequality, we perform the main calculation:
\begin{equation}
\begin{split}
&\left|f\star\left(\Delta^{\frac{1}{p'}}g\right)(x)\right| = \left|\int_G\ f(y)\cdot \Delta^{\frac{1}{p'}}\left(y^{-1}x\right)\cdot g\left(y^{-1}x\right)\ dy\right|\\
&\leqslant \int_G\ \bigg(|f|(y)^{\frac{p}{r}}\cdot |g|\left(y^{-1}x\right)^{\frac{q}{r}}\bigg) \cdot |f|(y)^{\left(1-\frac{p}{r}\right)} \cdot  \bigg(|g|\left(y^{-1}x\right)\bigg)^{\left(1-\frac{q}{r}\right)}\cdot \Delta^{\frac{1}{p'}}\left(y^{-1}x\right)\ dy \\
&\leqslant \bigg(\int_G\ |f|(y)^p \cdot |g|\left(y^{-1}x\right)^q\ dy\bigg)^{\frac{1}{r}}\cdot \bigg(\int_G\ |f|^{\left(1-\frac{p}{r}\right)q'}\ d\mu\bigg)^{\frac{1}{q'}}\cdot \bigg(\int_G\ |g|\left(y^{-1}x\right)^{\left(1-\frac{q}{r}\right)p'}\cdot \Delta\left(y^{-1}x\right)\ dy \bigg)^{\frac{1}{p'}} \\
&\stackrel{(\ref{coupleofequalities})}{=} \bigg(\int_G\ |f|(y)^p \cdot |g|\left(y^{-1}x\right)^q\ dy \bigg)^{\frac{1}{r}}\cdot \bigg(\int_G\ |f|^p\ d\mu \bigg)^{\frac{1}{q'}}\cdot \bigg(\int_G\ |g|\left(y^{-1}x\right)^q\cdot \Delta\left(y^{-1}x\right)\ dy \bigg)^{\frac{1}{p'}} \\
&\stackrel{y\mapsto xy}{=} \bigg(\int_G\ |f|(y)^p\cdot |g|\left(y^{-1}x\right)^q\ dy \bigg)^{\frac{1}{r}}\cdot \|f\|_p^{\frac{p}{q'}}\cdot \bigg(\int_G\ |g|\left(y^{-1}\right)^q\cdot \Delta\left(y^{-1}\right)\ dy \bigg)^{\frac{1}{p'}} \\
&=  \bigg(\int_G\ |f|(y)^p\cdot |g|\left(y^{-1}x\right)^q\ dy \bigg)^{\frac{1}{r}}\cdot \|f\|_p^{\frac{p}{q'}}\cdot \bigg(\int_G\ |g|^q\ d\mu \bigg)^{\frac{1}{p'}} \\
&= \bigg(\int_G\ |f|(y)^p\cdot |g|\left(y^{-1}x\right)^q\ dy \bigg)^{\frac{1}{r}}\cdot \|f\|_p^{\frac{p}{q'}}\cdot \|g\|_q^{\frac{q}{p'}} = \bigg(|f|^p\star|g|^q(x)\bigg)^{\frac{1}{r}}\cdot \|f\|_p^{\frac{p}{q'}}\cdot \|g\|_q^{\frac{q}{p'}}.
\end{split}
\label{mainyoungcalculation}
\end{equation}

\noindent
This leads to
\begin{equation*}
\begin{split}
\int_G\ \bigg|f\star\left(\Delta^{\frac{1}{p'}}g\right)(x)\bigg|^r\ dx &\leqslant \bigg(\int_G\ |f|^p\star|g|^q(x)\ dx\bigg)\cdot \|f\|_p^{\frac{pr}{q'}}\cdot \|g\|_q^{\frac{qr}{p'}}\\ 
&= \||f|^p\star|g|^q\|_1\cdot \|f\|_p^{\frac{pr}{q'}}\cdot \|g\|_q^{\frac{qr}{p'}}\\
&\leqslant \||f|^p\|_1 \cdot \||g|^q\|_1 \cdot \|f\|_p^{\frac{pr}{q'}}\cdot \|g\|_q^{\frac{qr}{p'}} = \|f\|_p^{p+\frac{pr}{q'}}\cdot \|g\|_q^{q+\frac{qr}{p'}},
\end{split}
\end{equation*}

\noindent
where the second inequality follows from Theorem 1.6.2 in \cite{DeitmarEchterhoff}, p. 26. Taking the $r$-th root, we conclude that 
$$\|f\star g\|_r \leqslant \|f\|_p^{\frac{p}{r}+\frac{p}{q'}}\cdot \|g\|_q^{q+\frac{qr}{p'}} = \|f\|_p\cdot \|g\|_q,$$

\noindent
which ends the proof.  
\end{proof}

\section{Main theorem}
\label{sectionmainpart}

Prior to demonstrating the main result of the paper we prove a simple lemma:

\begin{lemma}
Let $\Ffamily\subset L^p(G)$ be $L^p-$bounded and $L^p-$equicontinuous. If $\phi\in C_c(G)$ and $K\Subset G$ then $\Ffamily|_K\star \phi$ is relatively compact in $C_0(G)$ and 
$$\forall_{f\in \Ffamily}\ \supp\big(f|_K\star \phi\big) \subset K\cdot\supp(\phi).$$
\label{approximatingfamily}
\end{lemma}
\begin{proof}
Obviously, if $\Ffamily$ is $L^p-$bounded then so is $\Ffamily|_K.$ Furthermore, $\Ffamily|_K$ is trivially $L^p-$equivanishing, because it is supported in $K$. Last but not least, the $L^p-$equicontinuity of $\Ffamily$ implies the $L^p-$equicontinuity of $\Ffamily|_K.$

By Theorems \ref{Lpboundedness}, \ref{Lpequicont} and \ref{Lpequivanish} we conclude that $\Ffamily|_K\star \phi$ is bounded, equicontinuous and equivanishing in $C_0(G).$ By Theorem \ref{AAC0} we establish that $\Ffamily|_K\star \phi$ is relatively compact.

For the last part of the theorem, observe that for every $f\in\Ffamily$ we have
$$\forall_{x\in G}\ \big|f|_K\star\phi(x)\big| = \int_K\ \left|f(y)\cdot\phi\left(y^{-1}x\right)\right|\ dy.$$

\noindent
The integral on the right-hand side is 0 for $x \not\in K\cdot\supp(\phi),$ so
$$\forall_{f\in \Ffamily}\ \supp\big(f|_K\star \phi\big) \subset K\cdot\supp(\phi).$$

\noindent
This concludes the proof.
\end{proof}

At last, we have reached the climax of the paper. Let us recap how this main result relates to other compactness theorems known in the mathematical literature. To begin with, Theorem \ref{KolmogorovRieszWeilSudakovtheorem} generalizes the Kolmogorov-Riesz theorem\footnote{Apart from the papers \cite{HancheOlsenHolden}, \cite{Kolmogorov}, \cite{Riesz}, \cite{Tamarkin} mentioned in the Introduction, the proof of the Kolmogorov-Riesz theorem for $\reals^N$ can be found in \cite{Brezis}, p. 111 (Theorem 4.26) or in \cite{Precup}, p. 21 (Theorem 1.3).}, as $\reals^N$ is a locally compact Hausdorff group. Furthermore, we go to great lengths to render the proof below more transparent than the complicated Weil's argument\footnote{See \cite{Weil}, p. 53-54 for comparison.}. Last but not least, the final part of our result generalizes the Sudakov theorem\footnote{See \cite{HancheOlsenHoldenMalinnikova} or \cite{Sudakov}.}.

\begin{thm}
A family $\Ffamily \subset L^p(G)$ is relatively compact if and only if 
\begin{itemize}
	\item $\Ffamily$ is $L^p-$bounded,
	\item $\Ffamily$ is $L^p-$equicontinuous,
	\item $\Ffamily$ is $L^p-$equivanishing.
\end{itemize}

Furthermore, suppose that the group $G$ is such that for every open neighoburhood $U$ of the neutral element, there exists an element $x\in U$ such that $(x^n)_{n\in\naturals}$ is not contained in any compact set. Then the condition of $L^p-$boundedness is redundant. 
\label{KolmogorovRieszWeilSudakovtheorem}
\end{thm}
\begin{proof}
For the entire proof, which we divide into five steps, we fix $\eps>0.$

\vspace{0.3cm}
\noindent
\textbf{Step 1.} \textit{Relative compactness of $\Ffamily$ implies $L^p-$boundedness}

\vspace{0.3cm}
This step is immediate, as a compact set is always bounded in a metric space. 

\vspace{0.3cm}
\noindent
\textbf{Step 2.} \textit{Relative compactness of $\Ffamily$ implies $L^p-$equicontinuity}

\vspace{0.3cm}
Let $(f_n)_{n=1}^N$ be an $\frac{\eps}{3}-$net for the family $\Ffamily$. By Proposition 2.41 in \cite{FollandAHA}, p. 53 for every $n=1,\ldots,N$ there exists an open neighbourhood $U_n$ of the neutral element such that 
\begin{gather}
\sup_{x\in U_n}\ \|L_xf_n-f_n\|_p \leqslant \frac{\eps}{3} \hspace{0.4cm}\text{and}\hspace{0.4cm} \sup_{x\in U_n}\ \|R_xf_n-f_n\|_p \leqslant \frac{\eps}{3}.
\label{xinUnLxfnfn}
\end{gather}

\noindent
Put $U := \bigcap_{n=1}^N\ U_n$, which is obviously an open set. Consequently, for every $f\in\Ffamily$ there exists $n=1,\ldots,N$ such that 
\begin{gather*}
\sup_{x\in U}\ \|L_xf-f\|_p \leqslant \sup_{x\in U}\ \|L_xf-L_xf_n\|_p + \sup_{x\in U}\ \|L_xf_n-f_n\|_p + \|f_n-f\|_p \\
= 2 \|f_n-f\|_p + \sup_{x\in U}\ \|L_xf_n-f_n\|_p \stackrel{(\ref{xinUnLxfnfn})}{\leqslant} \eps.
\end{gather*}

\noindent
An analogous reasoning works for $\|R_xf-f\|_p$. We conclude that the family $\Ffamily$ is $L^p-$equicontinuous. 

\vspace{0.3cm}
\noindent
\textbf{Step 3.} \textit{Relative compactness of $\Ffamily$ implies $L^p-$equivanishing}

\vspace{0.3cm}
Let $(f_n)_{n=1}^N$ be an $\frac{\eps}{2}-$net for the family $\Ffamily$. For every $n=1,\ldots,N$ there exists $K_n\Subset G$ such that 
\begin{gather}
\int_{G\backslash K_n}\ |f_n|^p\ d\mu \leqslant \frac{\eps}{2}.
\label{GbackKnfp}
\end{gather}

\noindent
Put $K := \bigcup_{n=1}^N\ K_n$, which is obviously a compact set. Consequently, for every $f\in\Ffamily$ there exists $n=1,\ldots,N$ such that 
$$\int_{G\backslash K}\ |f|^p\ d\mu \leqslant \int_{G\backslash K}\ |f-f_n|^p\ d\mu + \int_{G\backslash K}\ |f_n|^p\ d\mu \stackrel{(\ref{GbackKnfp})}{\leqslant} \frac{\eps}{2} + \frac{\eps}{2} = \eps.$$

\noindent
This proves that $\Ffamily$ is $L^p-$equivanishing. 

\vspace{0.3cm}
\noindent
\textbf{Step 4.} \textit{$L^p-$boundedness, $L^p-$equicontinuity and $L^p-$equivanishing imply relative compactness of $\Ffamily$}

\vspace{0.3cm}
Due to Theorem \ref{choiceofphi} we pick $\phi\in C_c(G)$ such that 
\begin{gather}
\forall_{f\in \Ffamily}\ \|f\star \phi - f\|_p \leqslant \frac{\eps}{4}.
\label{fstarminusfepsilon3}
\end{gather}

\noindent
By $L^p-$equivanishing of $\Ffamily,$ there exists $K\Subset G$ such that 
\begin{gather}
\forall_{f\in \Ffamily}\ \|f-f|_K\|_p \leqslant \frac{\eps}{4\cdot \|\Delta^{-\frac{1}{p'}}\phi\|_1}.
\label{finalKchoice}
\end{gather}
 
\noindent
Due to Theorem \ref{approximatingfamily} we know that $\Ffamily|_K\star\phi$ is relatively compact in $C_0(G)$ and that 
$$\forall_{f\in\Ffamily}\ \supp(f|_K\star\phi) \subset D,$$

\noindent
where $D:=K\cdot\supp(\phi).$ By relative compactness of $\Ffamily|_K\star\phi$ there exists a finite sequence of functions $(g_n)_{n=1}^N\subset C_c(G)$ such that for every $f\in\Ffamily$ there exists $n=1,\ldots,N$ such that 
\begin{gather}
\big\|f|_K\star\phi - g_n\big\|_{\infty} \leqslant \frac{\eps}{4\cdot \mu(D)^{\frac{1}{p}}}.
\label{fstarfnstarepsilon3}
\end{gather}

\noindent
Due to inner regularity of $\mu$ there exists an open set $V\supset D$ such that 
\begin{gather}
\mu(V\backslash D) \leqslant \left(\frac{\eps}{4\cdot\max_{n=1,\ldots,N}\ \|g_n\|_{\infty}}\right)^p.
\label{usinginnerregularityofmu}
\end{gather}

\noindent
Since every locally compact group is normal\footnote{See Theorem 8.13 in \cite{HewittRoss}, p. 76.}, by Urysohn's lemma\footnote{See Theorem 1.5.11 in \cite{Engelking}, p. 41 or Lemma 4 in \cite{Kelley}, p. 115 or Theorem 33.1 in \cite{Munkres}, p. 207 or Theorem 1.5.6 in \cite{Pedersen}, p. 24.} there exists a function $u:G\longrightarrow [0,1]$ such that 
\begin{gather}
u|_D = 1 \hspace{0.4cm}\text{and}\hspace{0.4cm} u|_{G\backslash V} = 0.
\label{Urysohnfunction}
\end{gather}

\noindent
Finally, for every $f\in\Ffamily$ there exists $n=1,\ldots,N$ such that 
\begin{gather*}
\|f-h_n\cdot u\|_p \leqslant \|f-f\star\phi\|_p + \big\|f\star\phi - f|_K\star\phi\big\|_p + \big\|f|_K\star\phi-g_n\cdot u\big\|_p \\
\stackrel{(\ref{younginequality}),\ (\ref{fstarminusfepsilon3})}{\leqslant} \frac{\eps}{4} + \big\|f-f|_K\big\|_p\cdot \|\Delta^{-\frac{1}{p'}}\phi\|_1 + \bigg(\int_V\ \big|f|_K\star\phi - g_n\cdot u\big|^p\ d\mu\bigg)^{\frac{1}{p}}\\
\stackrel{(\ref{finalKchoice}),\ (\ref{Urysohnfunction})}{\leqslant} \frac{\eps}{2} + \bigg(\int_D\ \big|f|_K\star\phi - g_n\big|^p\ d\mu\bigg)^{\frac{1}{p}} + \bigg(\int_{V\backslash D}\ |g_n|^p\ d\mu\bigg)^{\frac{1}{p}}\\
\leqslant \frac{\eps}{2} + \|f|_K\star\phi - g_n\|_{\infty}\cdot \mu(D)^{\frac{1}{p}} + \|g_n\|_{\infty}\cdot \mu(V\backslash D)^{\frac{1}{p}}  \stackrel{(\ref{fstarfnstarepsilon3}),\ (\ref{usinginnerregularityofmu})}{\leqslant} \eps.
\end{gather*}

\noindent
This demonstrates that $(g_n\cdot u)_{n=1}^N$ is an $\eps-$net for $\Ffamily.$ Thus $\Ffamily$ is relatively compact.

\vspace{0.3cm}
\noindent
\textbf{Step 5.} \textit{Sudakov's part}

\vspace{0.3cm}
We will now prove that $L^p-$boundedness follows from $L^p-$equicontinuity and $L^p-$equivanishing under the assumption that for every open neighoburhood $U$ of the neutral element, there exists an element $x\in U$ such that $(x^n)_{n\in\naturals}$ is not contained in any compact set. By $L^p-$equicontinuity of $\Ffamily$ there exists an open neighbourhood $U_e$ of the neutral element such that 
\begin{gather}
\forall_{\substack{x\in U_e\\ f\in\Ffamily}}\ \left(\int_G\ |L_xf-f|^p\ d\mu\right)^{\frac{1}{p}} \leqslant 1.
\label{Lpequicontinuityand1}
\end{gather} 

\noindent
Furthermore, due to $L^p-$equivanishing there exists $K\Subset G$ such that 
\begin{gather}
\forall_{f\in\Ffamily}\ \left(\int_{G\backslash K}\ |f|^p\ d\mu\right)^{\frac{1}{p}} \leqslant 1.
\label{Lpequivanishingand1}
\end{gather}

\noindent
Let $x_*\in U_e$ be an element such that $(x_*^n)_{n\in\naturals}$ is not contained in any compact set. For every $f\in\Ffamily$ we have
\begin{gather*}
\left(\int_K\ |f|^p\ d\mu\right)^{\frac{1}{p}} \leqslant \left(\int_K\ |L_{x_*}f - f|^p\ d\mu\right)^{\frac{1}{p}} + \left(\int_K\ |L_{x_*}f|^p\ d\mu\right)^{\frac{1}{p}} \\
\stackrel{(\ref{Lpequicontinuityand1})}{\leqslant} 1 + \left(\int_G\ |f(x_*y)|^p\cdot \mathds{1}_K(y)\ dy\right)^{\frac{1}{p}} = 1 + \left(\int_{x_*K}\ |f|^p\ d\mu\right)^{\frac{1}{p}}.
\end{gather*}

\noindent
Using an inductive reasoning  we have
\begin{gather}
\forall_{n\in\naturals}\ \left(\int_K\ |f|^p\ d\mu\right)^{\frac{1}{p}} \leqslant n + \left(\int_{x_*^nK}\ |f|^p\ d\mu\right)^{\frac{1}{p}}.
\end{gather}

Observe that there exists $N\in\naturals$ such that $x_*^NK\cap K = \emptyset,$ since otherwise we would have $(x_n)_{n\in\naturals}\subset K\cdot K^{-1}$, contrary to our assumption. Finally, we have
$$\|f\|_p \leqslant \left(\int_K\ |f|^p\ d\mu\right)^{\frac{1}{p}} + \left(\int_{G\backslash K}\ |f|^p\ d\mu\right)^{\frac{1}{p}} \leqslant N + \left(\int_{x_*^NK}\ |f|^p\ d\mu\right)^{\frac{1}{p}} + 1 \stackrel{(\ref{Lpequivanishingand1})}{\leqslant} N+2,$$

\noindent
which ends the proof.
\end{proof}

\begin{cor}
A family $\Ffamily \subset L^p(\reals)$ is relatively compact if and only if 
\begin{itemize}
	\item $\Ffamily$ is $L^p-$equicontinuous,
	\item $\Ffamily$ is $L^p-$equivanishing.
\end{itemize}
\end{cor}
\begin{proof}
The corollary follows from Theorem \ref{KolmogorovRieszWeilSudakovtheorem} and the fact that if $U_0$ is an open neighbourhood of the neutral element (namely $0$) and if $x\in U_0,\ x\neq 0$ then the sequence $(n\cdot x)_{n\in\naturals}$ is not contained in any compact subset of $\reals.$
\end{proof}

As a final remark let us consider the space $\ell^p(\integers)$. Its relatively compact subsets are characterized as follows\footnote{See \cite{HancheOlsenHolden}.}:

\begin{thm}
A family $\Ffamily \subset \ell^p(\integers)$ is relatively compact if and only if
\begin{itemize}
	\item $\Ffamily$ is $\ell^p-$bounded,
	\item $\Ffamily$ is $\ell^p-$equivanishing, i.e. for every $\eps>0$ there exists $N\in\naturals$ such that
	$$\forall_{x\in\Ffamily}\ \sum_{|n| > N}\ |x_n|^p \leqslant \eps.$$
\end{itemize}
\label{discreteKR}
\end{thm}  

Firstly, let us observe that Theorem \ref{discreteKR} does not mention ``$\ell^p-$equicontinuity'' $-$ this is because in $\ell^p(\integers)$ every family is $\ell^p-$equicontinuous. Furthermore, we argue that $\ell^p-$boundedness is not a redundant condition in characterizing relatively compact families in $\ell^p(\integers)$. This does not contradict the final part of Theorem \ref{KolmogorovRieszWeilSudakovtheorem}, because the singleton $\{0\}$ is an open neighbourhood of the neutral element (namely $0$) and the sequence $(n\cdot 0)_{n\in\naturals}$ (the only possible sequence we can construct from the single element of $\{0\}$) is contained in the compact set $\{0\}$. 

Let $x \in \ell^p(\integers)$ be such that $x_0 = 1$ and $x_n = 0$ for $n\in\integers\backslash\{0\}.$ Consider the family $\Ffamily = (n\cdot x)_{n\in\naturals}$. Obviously, the family $\Ffamily$ is $\ell^p-$equivanishing, but it is not $\ell^p-$bounded. Furthermore, it is not relatively compact which proves that $\ell^p-$boundedness is indeed indispensible for the Theorem \ref{discreteKR} to work.

\section*{Acknowldegements}

I would like to express my gratitude towards Wojciech Kryszewski, whose insightful questions forced me to repeatedly rethink the ideas that I have been working on. It is a privilege to have such a great source of constructive criticism.

I also wish to thank Robert Sta\'nczy for the reference to the Arzel\`a-Ascoli theorem for $C_0(X)$ in the monograph ``Integral Equations and Stability of Feedback Systems'' by Constantin Corduneanu.

\end{document}